\documentclass[11pt,titlepage]{article}

\usepackage{graphicx,wrapfig,eso-pic,float,epsfig}
\usepackage{amsmath,amssymb,mathrsfs,bm,amsthm}
\usepackage{relsize}
\newtheorem{thrm}{Theorem}
\newtheorem{lem}{Lemma}
\newtheorem{cor}{Corollary}
\DeclareMathAlphabet{\mathpzc}{OT1}{pzc}{m}{it}
\newcommand{\dl}{\delta}
\newcommand{\lt}{\left}
\newcommand{\rt}{\right}
\newcommand{\ra}{\rightarrow}
\newcommand{\RA}{\Rightarrow}
\newcommand{\tb}{\textbf}
\newcommand{\ti}{\textit}
\newcommand{\fa}{\forall}
\newcommand{\N}{\mathbb{N}}
\newcommand{\R}{\mathbb{R}}
\newcommand{\A}{\mathcal{A}}
\newcommand{\LL}{\mathcal{L}}
\newcommand{\C}{\mathcal{C}}
\newcommand{\F}{\mathcal{F}}

\begin{document}
\begin{center}
	\Large{\boldmath{$\mathlarger{\mathlarger{\pi}}$} \bf{and Arc-Length}}\\
{Joseph Amal Nathan\\ Reactor Physics Design Division, BARC, Mumbai - 400085 INDIA}
\end{center}
\tb{Abstract:} We use the classical definitions (i) $\pi$ is the ratio of area to the square of the radius of a circle; (ii) $\pi$ is the ratio of circumference to the diameter of a circle, to prove $\pi$'s existence within the purview of Euclidean geometry. Next we show that the ``arc-length" (Definition 1) is deducible from Euclidean geometry. Then we prove the Non-Euclidean-Axioms(NEA) of Archimedes (Corollary 4 and 5) and that the arc-length integral converges to the arc-length. We justify why `Euclidean Metric' (Definition 5) is a correct metric for arc-length; derive expressions for area, circumference of a circle and finally prove the equivalence of definitions (i) and (ii).\\\\
\tb{Keyword:} Euclidean geometry, $\pi$, bounded sequences, arc-length.\\
\tb{MSC[2010]:} 03F60, 26E40, 28A75, 40A05.

\section{Introduction}
In 250 BC Archimedes had to use NEA to prove the existence of $\pi$. In Calculus the arc-length formula and a proof of `$\sin x/x \ra 1$ as $x \ra 0$' are based on one of these NEA. Hence questions like: ``Is $\pi$ a constant independent of NEA? Are definitions (i) and (ii) equivalent? Is there a justification that the arc-length integral converges to the arc-length?" are still being asked. Yet the NEA have remained unanalysed till today, may be that the belief `arc-length cannot be deduced from Euclidean geometry' was strong. Though the popular article\cite{JS} and its references, gives an idea about discussions on $\pi$, only in references\cite{EA}, \cite{LG} the existence of $\pi$ is addressed. Alternately $\pi$ is defined using $\int^{+1}_{-1}\sqrt{1-x^2}dx$\cite{MS} and circular functions in Real Analysis\cite{WR}, but still it does not help resolve the above queries. Realising that only the classical definitions can provide answers and motivated by $\pi$'s existence proof in real analysis, we revisit definitions (i) and (ii) to find a proof.
\section{Construction}
Refer to Fig-1. We adopt a modified construction of Archimedes. With $r$ as radius and $O$ as center draw a circle. In this circle inscribe a regular $k-$sided polygon with side $L_0$. At all vertices of the inscribed polygon draw tangents to the circle such that adjacent tangents meet to form a circumscribed polygon, for example tangents drawn at the vertices $P,~Q$ of the inscribed polygon meet to form vertex $V$ of the circumscribed polygon. This circumscribed polygon will be regular with $k$ sides and let $U_0$ be the side. Then we have two regular $k$-sided polygons inscribing and circumscribing the circle with perimeters $kL_0$ and $kU_0$ respectively. This is the zeroth stage. In Fig-1 since $k=6$ we have $L_0=r$ and for $k>6$, $L_0 <r$. Again with $O$ as center consider another circle with radius $r'$ and let $L_0'$ be the side of it's zeroth stage regular $k-$sided inscribed polygon. Then from property of similar triangles $L_0 \times r^\prime=L^\prime _0 \times r$, which implies $L_0 \propto r$.
\begin{wrapfigure}{ht}{0.68\textwidth}
	\centering
	\includegraphics[width=0.64\textwidth, trim={10mm 8mm 2mm 8mm}]{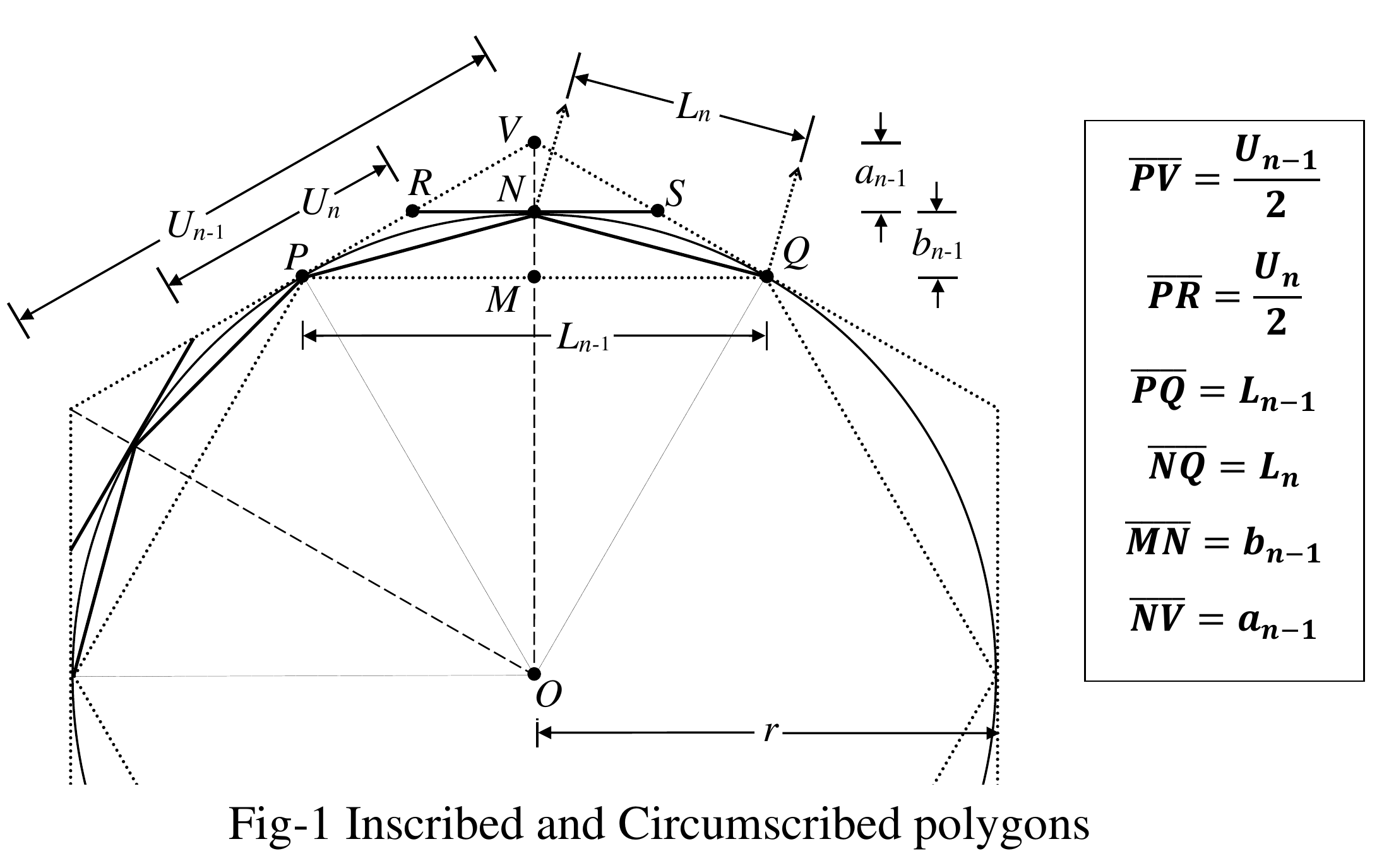}
\end{wrapfigure}
For the next (first) stage draw a line from $O$ to a vertex $V$ of the circumscribed polygon. This line will bisect $\overline{PQ}$ of the inscribed polygon at $M$ and bisect $\hbox{Arc}(PQ)$ of the circle at $N$. Join $P$ and $N$ and similarly $N$ and $Q$. At $N$ draw a tangent to the circle meeting the circumscribed polygon at $R$ and $S$. Then segments $\overline{PN},~\overline{NQ}$ will be a part of the next (first) stage inscribed polygon and segments $\overline{PR},~\overline{RN},~\overline{NS},~\overline{SQ}$ will be a part of the next (first) stage circumscribed polygon. This exercise is repeated at all the remaining $(k-1)$ vertices to get at the next (first) stage two regular $2k$-sided polygons inscribing and circumscribing the circle with side $L_1$, perimeter $2kL_1$ and side $U_1$, perimeter $2kU_1$ respectively. This completes the next (first) stage. When we keep repeating this process at the end of the $n^{th}$ stage we will have two regular polygons of $2^nk$ sides inscribing and circumscribing the circle with side $L_n$, perimeter $P_n^L=2^nkL_n$ and side $U_n$, perimeter $P_n^U=2^nkU_n$ respectively.
\section{Area}
Fig-1 shows the procedure to get the $n^{th}$ stage polygon from the $(n-1)^{th}$ stage polygon. Taking $n=1$ will show the procedure from zeroth to the first stage. In this construction using Pythagoras theorem the following equations can be easily seen satisfied,
\begin{align}
b_{n-1}^2 &= L_n^2 - \lt( \frac{L_{n-1}}{2} \rt)^2, \label{b}\\
(r-b_{n-1})^2 &= r^2 - \lt( \frac{L_{n-1}}{2} \rt)^2, \label{rb}\\
a_{n-1}^2 &= \lt( \frac{U_{n-1}}{2}-\frac{U_{n}}{2} \rt)^2 - \lt( \frac{U_{n}}{2} \rt)^2 = \frac{U_{n-1}}{4}(U_{n-1}-2U_n), \label{a}\\
(r+a_{n-1})^2 &= r^2+ \lt( \frac{U_{n-1}}{2} \rt)^2, \label{ra}\\
(a_{n-1}+b_{n-1})^2 &= \lt( \frac{U_{n-1}}{2} \rt)^2-\lt( \frac{L_{n-1}}{2} \rt)^2. \label{ab}
\end{align} 
Expanding $(r+a_{n-1})^2=[(r-b_{n-1})+(b_{n-1}+a_{n-1})]^2$ and $b_{n-1}^2=[r-(r-b_{n-1})]^2$ and substituting \eqref{b} \eqref{rb}, \eqref{ra}, \eqref{ab} gives
\begin{align*}
\lt( \frac{L_{n-1}}{2} \rt)^4 &= (r-b_{n-1})^2(b_{n-1}+a_{n-1})^2=\lt[ r^2 - \lt( \frac{L_{n-1}}{2} \rt)^2 \rt] \lt[ \lt( \frac{U_{n-1}}{2} \rt)^2-\lt( \frac{L_{n-1}}{2} \rt)^2 \rt],\\
(L_n^2-2r^2)^2 &= 4r^2(r-b_{n-1})^2=4r^2 \lt[ r^2 - \lt( \frac{L_{n-1}}{2} \rt)^2 \rt],
\end{align*}
respectively, which after simplification becomes
\begin{align}
\lt(\frac{U_{n-1}}{2r} \rt) &= \frac{\lt(\frac{L_{n-1}}{2r} \rt)}{\sqrt{1-\lt( \frac{L_{n-1}}{2r} \rt)^2}}, \label{eq:U}\\
2 \lt(\frac{L_n}{2r} \rt)^2 &= 1-\sqrt{1-\lt( \frac{L_{n-1}}{2r} \rt)^2}, \label{eq:L}
\end{align}
respectively. Using above equations we will construct a proof similar to the existence of $e$ in \cite{CG} or of Riemann integral\cite{CG}. Let $A_n^U$ and $A_n^L$ be the areas of the circumscribed and the inscribed polygons. Then,
\begin{lem} \label{L1} The sequences \emph{(i)} $\{ U_n \},~\{ L_n \},~\{ a_n \},~\{ b_n \}$ are strictly decreasing and null, \emph{(ii)} $\{P_n^L\},~\{A_n^L\}$ are strictly increasing, \emph{(iii)} $\{P_n^U\},~\{A_n^U\}$ are strictly decreasing and \emph{(iv)} $\{P_n^L/2r\}$, $\{P_n^U/2r\}$ have values independent of $r$.
\end{lem}
\begin{proof} Eliminating $(L_{n-1}/2)^2$ from \eqref{b}, \eqref{rb} and  $(U_{n-1}/2)^2$ from \eqref{a}, \eqref{ra} gives $2rb_{n-1}$
$=L_n^2$ and $4ra_{n-1}=U_{n-1}U_n$ respectively. From \eqref{a} we have $U_{n-1} >U_n$, using this in \eqref{eq:U} implies $L_{n-1} >L_n$. Let $l=\hbox{GLB}(\{ L_n \})$, since $\{ L_n \}$ is bounded below by zero and for sufficiently large $n$, $L_n < r$ we have $0 \leq l < r$. Let $k_n=L_{n-1}/L_n$ then from \eqref{b} $L_{n-1} < 2L_n$ implies $1 < k_n < 2$. Substituting for $L_{n-1}$ in \eqref{eq:L} and simplifying gives $L_n[L_n^2-(4-k_n^2)r^2]=0$. If $l \neq 0$ we have $L_n= (\sqrt{4-k_n^2})r$ and as $n \ra \infty$, $L_n \ra l \RA k_n \ra 1$. Then for sufficiently large $n$, $\sqrt{4-k_n^2} >1 \RA L_n >r$ a contradiction. This shows $\{ L_n \}$, from (6) $\{ U_n \}$ and hence $\{ a_n \}$, $\{ b_n \}$ all are strictly decreasing null sequences.\\
From \eqref{ab}, \eqref{b}, \eqref{a} we get $L_{n-1}<U_{n-1}$, $2L_n>L_{n-1}$, $2U_n<U_{n-1}$ respectively. Multiplying these three inequalities by $2^{n-1}k$ gives $P^L_{n-1} < P^U_{n-1}$, $P^L_n > P^L_{n-1}$, $P^U_n < P^U_{n-1}$ respectively. From the later two inequalities we get $\{ P^L_n \}$ and $\{ P^U_n \}$ are strictly increasing and decreasing respectively. As $L_0 \propto r$, we see that $L_0/2r$ and from \eqref{eq:L} $L_n/2r~\fa~ n \in \N$, are independent of $r$. Similarly from \eqref{eq:U} $U_0/2r$ and $U_n/2r~\fa~ n \in \N$ are values independent of~$r$. Then all values in $\{P_n^L/2r\}$ and $\{P_n^U/2r\}$ are independent of~$r$. From Fig-1 area of $\triangle ONQ=r L_{n-1}/4$ then the area of the $n^{th}$ stage inscribed polygon is $A^L_{n-1}=2^nkr L_{n-1}/4=rP_{n-1}^L/2$. Similarly area $\square OPVQ=r U_{n-1}/2$ then the area of the $(n-1)^{th}$ stage circumscribed polygon is $A^U_{n-1}=2^{n-1}kr U_{n-1}/2=rP_{n-1}^U/2$. Hence $\{A_n^L\}$ is strictly increasing and $\{A_n^U\}$ is strictly decreasing.
\end{proof}
\noindent Let $\A$ denote the area of the circle with radius $r$ then,
\begin{thrm} \label{T1} 
	$\A/r^2=$\emph{LUB(}$\{ P^L_n/2r \}$\emph{)}$=$\emph{GLB(}$\{ P^U_n/2r \}$\emph{)} is a constant, independent of r.
\end{thrm}
\begin{proof}
	Let $\dl^L_{n-1}= 2^{n-1}k \{$Area between $\overline{PQ}$ and $\hbox{Arc}(PQ) \}$ and $\dl^U_{n-1}=2^{n-1}k \{$Area between $\hbox{Arc}(PN)$ and $\overline{PV}$ + Area between $\hbox{Arc}(NQ)$ and $\overline{VQ} \}$. Then from Fig-1 we have $A^L_{n-1}+\dl^L_{n-1}=\A =A^U_{n-1}-\dl^U_{n-1}$ and $\dl_{n-1}^L, \dl_{n-1}^U \leq (a_{n-1}+b_{n-1})L_{n-1}/2$, which for the $n^{th}$ stage is,
	\begin{align}
	A^L_n+\dl ^L_n&=\A =A^U_n-\dl ^U_n \hbox{~~~and~~~}\dl_n^L, \dl_n^U \leq \frac{1}{2}(a_n+b_n)L_n. \label{A}
	\end{align}
	Since, by Lemma \ref{L1}, $\{A_n^L\}$ is strictly increasing, $\{A_n^U\}$ is strictly decreasing and $\{ L_n \},~\{ a_n \},~\{ b_n \}$ are null, both $\{\dl_n^L\}$ and $\{\dl_n^U\}$ are strictly decreasing null sequences. So for any $n \in \N, ~A^L_n < \A < A^U_n$. Then  LUB($\{ A^L_n \}$) = GLB($\{ A^U_n \}$) = $\A$. This proves the existence of $\A$. Now to show $\A$ is unique, define $I_n=[A^L_n,A^U_n]$ then $\A \in I_n~\fa~n$. Let there exist a $\A^\prime \neq \A$ such that $\A^\prime \in I_n~\fa~n$. Since $\{ A^U_n - A^L_n \}$ is a null sequence, there exists an integer $j$ such that $A^U_j-A^L_j<|\A^\prime-\A| \RA \A^\prime \notin I_j$ a contradiction. Since $A^L_n/r^2=P_n^L/2r$ and $A^U_n/r^2=P_n^U/2r$, LUB($\{P_n^L/2r\}$)=GLB($\{P_n^U/2r\}$)=$\A/r^2$, from Lemma \ref{L1} a constant independent of $r$ and conventionally denoted as $\pi$.
\end{proof}
\begin{cor}
	The area of the circle with radius r is $\pi r^2$.
\end{cor}
\begin{proof}
	From Theorem \ref{T1} we have $\A/r^2=\pi \RA \A=\pi r^2$.
\end{proof}
\section{Circumference and Arc-Length}
In this section `\ti{curve}' means the following,\\
\tb{Definition 1.} \ti{A \tb{curve} is a continuous function $f:[a,b] \ra \R$, such that tangent exists for $f(x)$ on $[a,b]$.}\\
\tb{Definition 2.} \ti{Distance between two points along a section of a curve is defined as \tb{arc-length}.}\\
Let $\Omega=\{a=x_0<x_1<...<x_{n-1}<x_n=b \}$, $n \geq 1$ be a partition of $[a,b]$ and $\Gamma(\Omega)= \underset{j=1}{\overset{n}{\sum}} \sqrt{(x_j-x_{j-1})^2+(f(x_j)-f(x_{j-1}))^2}$.\\
\tb{Definition 3.} \ti{If there exists a positive number $M$ such that for all possible partitions $\Omega$, $\Gamma(\Omega) < M$, then the curve $f(x)$ is said to be \tb{rectifiable} on $[a,b]$.}\\
\tb{Definition 4.} \ti{Determining the arc-length of a curve is called \tb{rectification}.}\\
For points $(x_1,y_1)$ and $(x_2,y_2)$ in a plane,\\
\tb{Definition 5.} \ti{\tb{Euclidean metric} is defined as $\sqrt{(x_1-x_2)^2+(y_1-y_2)^2}$.}\\
\tb{Definition 6.} \ti{\tb{Taxicab metric}}\cite{EK} \ti{is defined as $|x_1-x_2|+|y_1-y_2|$.}\\\\
Archimedes gets bounds for $\pi$ by showing that the circumference $\mathcal C$ of a circle lies between the perimeters of the inscribed and circumscribed polygon, which followed from convexity and the NEA. With these he shows that the circle is rectifiable. For details refer to the book by Thomas L. Heath\cite{TH}. Archimedes definition of convexity and the proposed axioms can be considered as the first rigorous attempt for the rectification of a curve\cite{DR}. 

\begin{wrapfigure}{h}{0.50\textwidth}
	\centering	
	\includegraphics[width=0.48\textwidth, trim={6mm 6mm 0 2mm}]{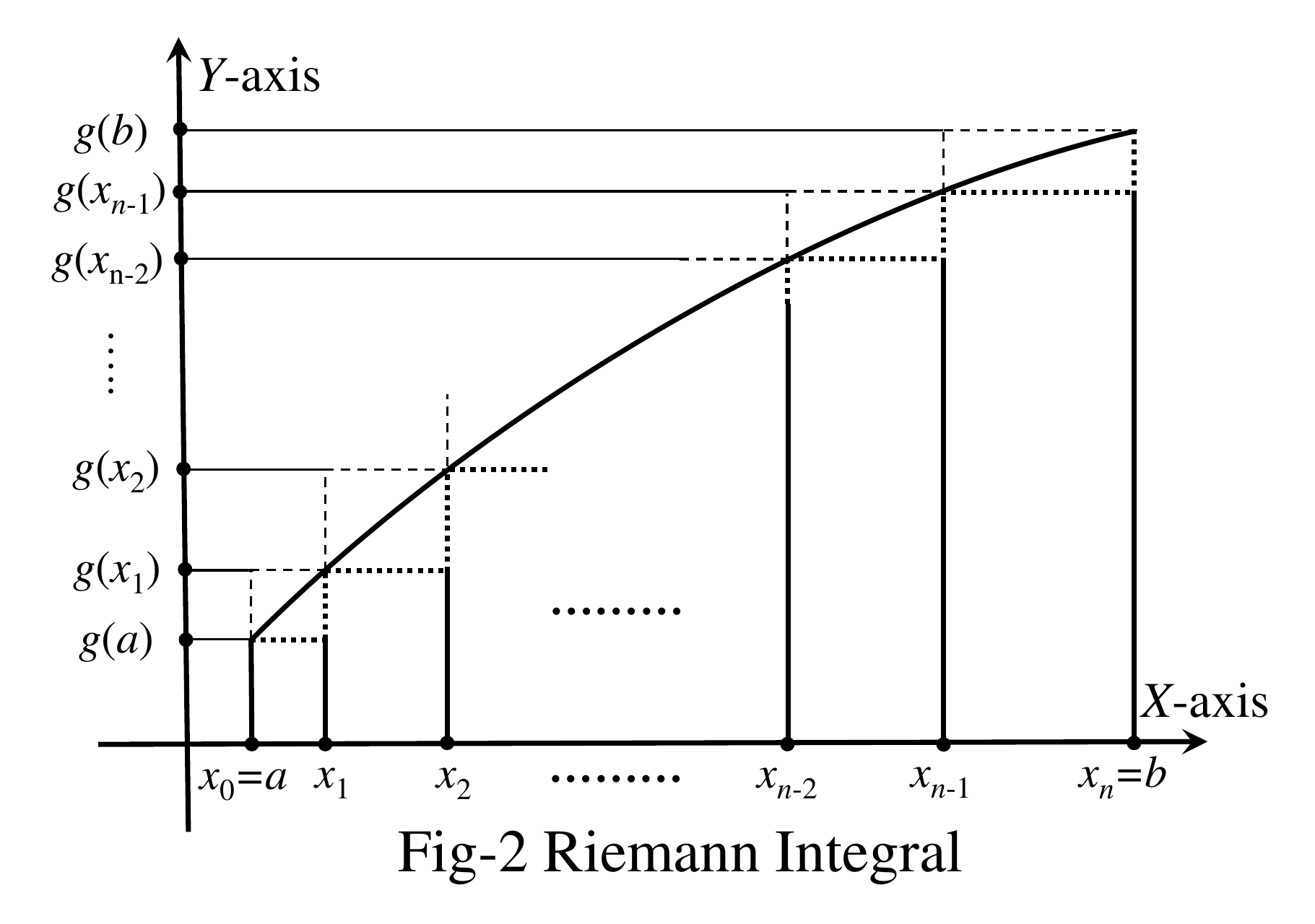}
\end{wrapfigure}

\noindent The rectification of a curve may also be achieved with a metric of proper choice, as all metrics may not lead to the rectification. Take for example the Riemann integral of a strictly increasing function $g(x)$ continuous on $[a,b]$ as shown in Fig-2. The upper and lower Riemann sum evaluated from the area of the upper and lower rectangles converges to the area under $g(x)$, because of their dependency on the partitions on $[a,b]$. Now to find the arc-length of $g(x)$ when we use the Taxicab metric and add the perimeters of lower and upper rectangles separately along $g(x)$, both sums always gives the constant $|b-a|+|g(b)-g(a)|$ irrespective of the partitions on $[a,b]$ and hence never converges to the arc-length of $g(x)$.

\begin{wrapfigure}{h}{0.50\textwidth}
	\centering	
	\includegraphics[width=0.48\textwidth, trim={6mm 8mm 0 4mm}]{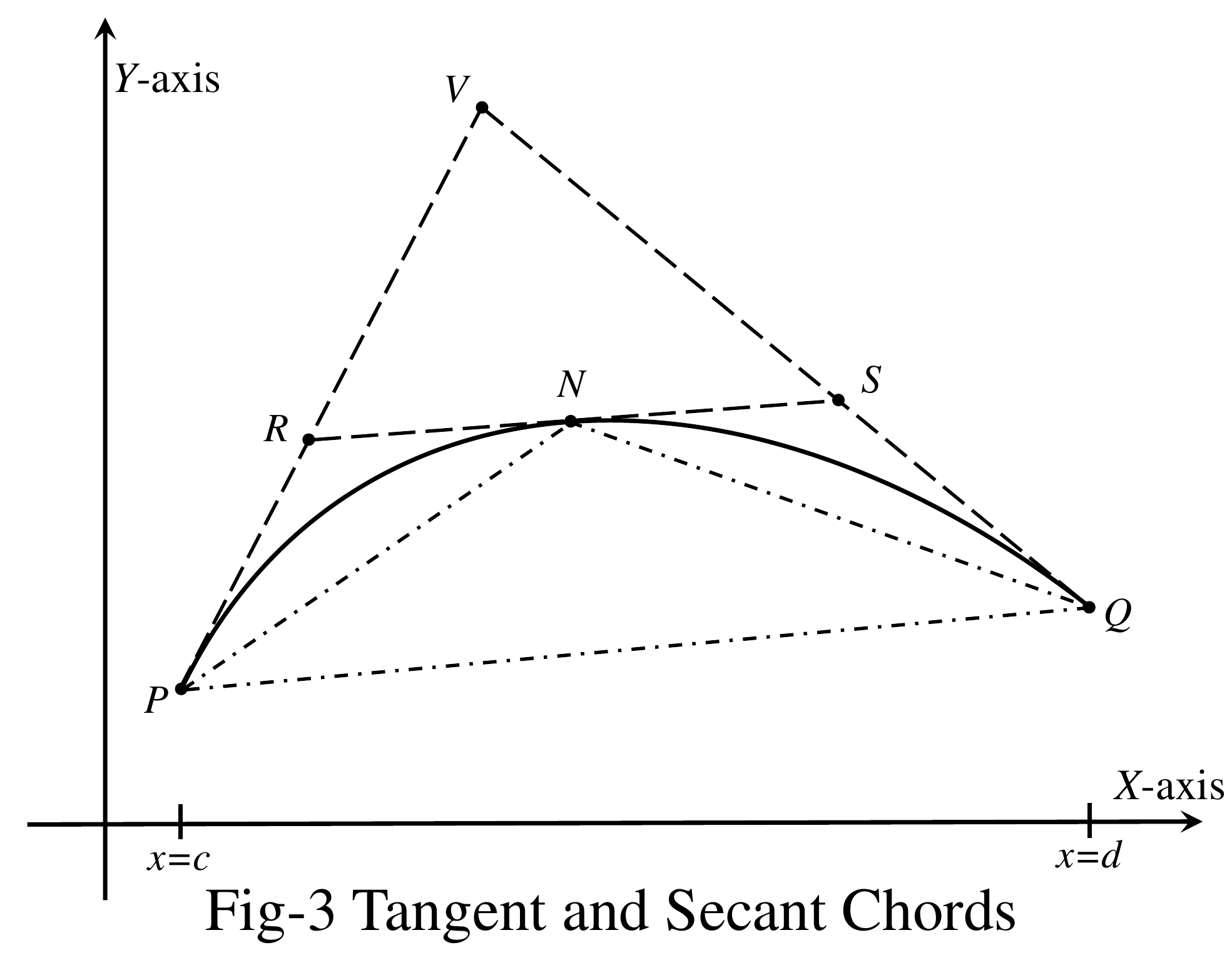}
\end{wrapfigure}

\noindent Inspired by the circle (Fig-1) where we used the Euclidean metric along with the inscribed and circumscribed polygons, we show that for $f(x)$ the use of Euclidean metric along with a similar construction leads to the rectification. Unlike the Taxicab metric which always give constant sequences, we find that the use of Euclidean metric leads to strictly increasing and decreasing sequences which converges to the arc-length of $f(x)$. We call the section of $f(x)$ in a closed interval which is either convex or concave as `\ti{cavex segment}'. Let the straight line segment from the point of contact of a tangent on a cavex segment to the intersection point of the tangent to its adjacent tangent be termed tangent-chord and to differentiate we will term the `chord' as secant-chord. Refer to Fig-3, for cavex segment $\hbox{Arc}(PQ)$ while $\overline{PQ}$, $\overline{PN}$, $\overline{NQ}$ are secant-chords, $\overline{PV}$, $\overline{VQ}$, $\overline{PR}$, $\overline{RN}$, $\overline{NS}$, $\overline{SQ}$ are tangent-chords. Let $I=[c,d]$ be the domain of the cavex segment of $f(x)$. Let $\Omega_k=\{c=x_0<x_1<...<x_{k-1}<x_k=d \}$, $k \geq 1$ be a partition of $I$. For $j=0, \cdots ,k$, let $y_j=f(x_j)$. We call the partition $\Omega_k+p$ of $I$ where $p \in (I \setminus \Omega_k)$ as a single point refinement of $\Omega_k$. For $j=1, \cdots ,k$, draw tangents to $f(x)$ at $(x_{j-1},y_{j-1})$ and $(x_j,y_j)$ such that they intersect and let $(X_j,~Y_j)$ be their coordinates. Let $s_j(\Omega_k)=\sqrt{(x_j-x_{j-1})^2+(y_j-y_{j-1})^2}$ and $t_j(\Omega_k)=\sqrt{(X_j-x_{j-1})^2+(Y_j-y_{j-1})^2}+\sqrt{(x_j-X_j)^2+(y_j-Y_j)^2}$. Then we define,
\begin{equation}
S(\Omega_k)=\underset{j=1}{\overset{k}{\sum}} s_j \hbox{~~and~~} T(\Omega_k)=\underset{j=1}{\overset{k}{\sum}} t_j\label{stm}
\end{equation}
as the secant-measure and tangent-measure of the arc-length of $f(x)$ on $I$ for the partition $\Omega_k$ respectively.

\begin{lem} \label{L2} Let $\Omega$ be any partition of a closed interval $I$ of a cavex segment, then the secant-measure $S(\Omega)$ is less than the tangent-measure $T(\Omega)$. If $\Omega_1 \subset \Omega_2 \cdots \subset \Omega_{n-1} \subset \Omega_n \cdots$ are infinite partitions on $I$, then sequences $\{ S(\Omega_n) \}$ and $\{ T(\Omega_n) \}$ are strictly increasing and decreasing respectively.
\end{lem}
\begin{proof}
For any partition $\Omega$ since $(X_j,~Y_j)$ does not lie on the secant-chord, from triangle inequality we get $S(\Omega) < T(\Omega)$. Now any partition $\Omega_j$ can be achieved from $\Omega_{j-1}$ through a finite sequence of single point refinements. Refer to Fig-3 in $\hbox{Arc}(PQ)$, $P$ and $Q$ are adjacent points. The tangent-chords $\overline{PV}$ and $\overline{VQ}$ along with the secant-chord $\overline{PQ}$ forms $\triangle PVQ$. Then length of $\overline{PQ}$ will be the secant-measure and length of $\overline{PV}+\overline{VQ}$ will be the tangent-measure for $\hbox{Arc}(PQ)$. Any single point refinement will give a point, say $N$, between points $P$ and $Q$ on $\hbox{Arc}(PQ)$, hence will lie in the interior of $\triangle PVQ$. A tangent at $N$ will intersect $\overline{PV}$ and $\overline{VQ}$ at $R$ and $S$ respectively. Now the secant-measure will be the sum of lengths of two secant-chords $\overline{PN}$ and $\overline{NQ}$ and the tangent-measure will be sum of lengths of four tangent-chords $\overline{PR}$, $\overline{RN}$, $\overline{NS}$ and $\overline{SQ}$. Then it follows from triangle inequality, the secant-measure will increase and the tangent-measure will decrease. So $S(\Omega_{n-1}) < S(\Omega_n)$ and $\{ S(\Omega_n) \}$ will be a strictly increasing sequence. Similarly $T(\Omega_{n-1}) > T(\Omega_n)$ and $\{ T(\Omega_n) \}$ will be a strictly decreasing sequence.
\end{proof}
\begin{lem} \label{L3} As two points on a cavex segment approach each other, the section of the cavex segment between the two points tends to a straight line.
\end{lem}

\begin{wrapfigure}{h}{0.55\textwidth}
	\centering	
	\includegraphics[width=0.53\textwidth, trim={6mm 8mm 4mm 10mm}]{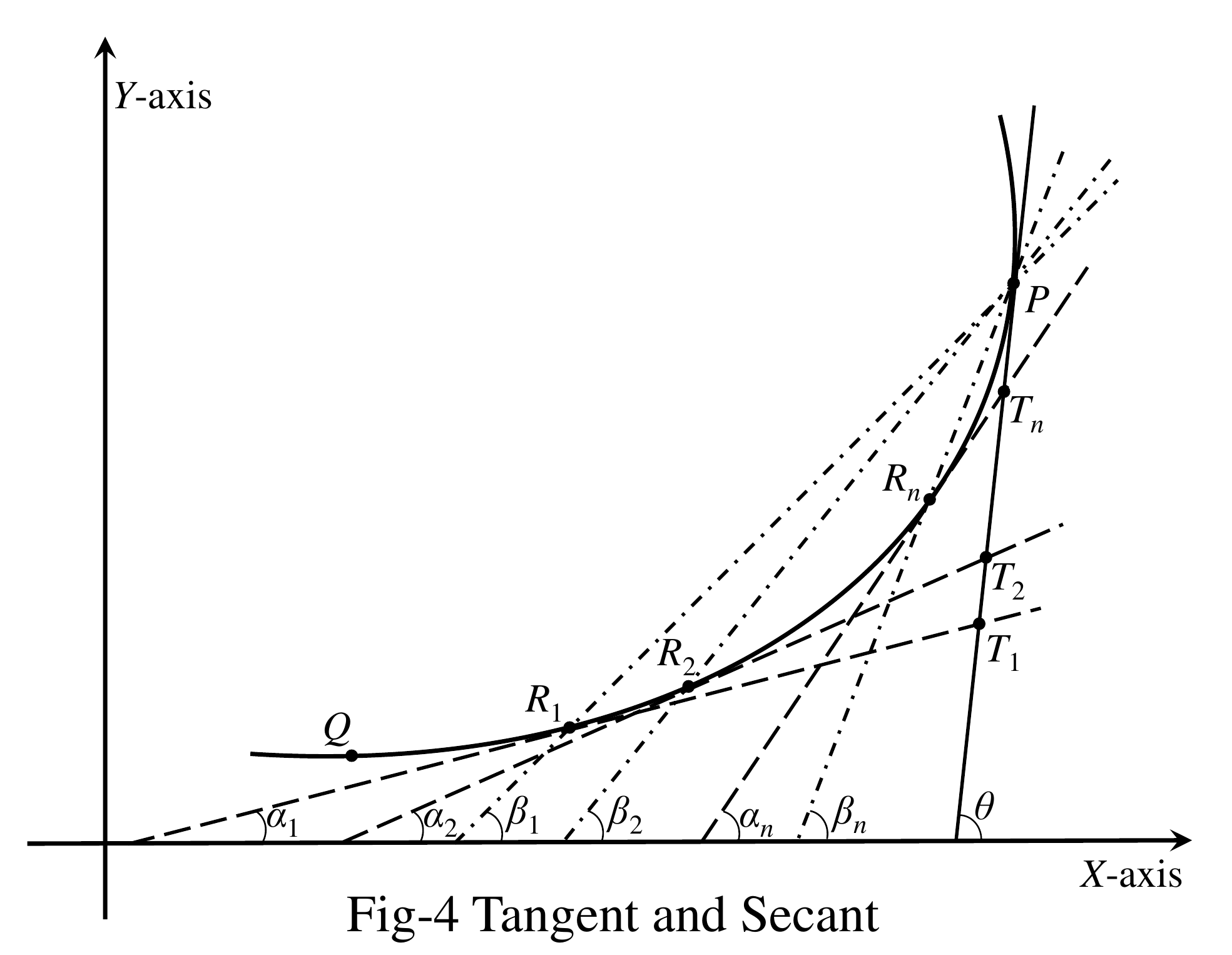}
\end{wrapfigure}

\noindent \ti{Proof.} There is no loss of generality in considering the cavex segment $\hbox{Arc}(PQ)$ in Fig-4. Let the tangent at $P$ make an angle $\theta$ with the $X-$axis. Let $R$ be a point on $\hbox{Arc}(PQ)$ and $R_1,R_2, \ldots , R_n, \ldots$ represent different positions of $R$ such that for $i \in \N, R_{i+1}$ is nearer to $P$ than $R_i$. Let the tangent at $R_i$ and the secant through the points $P$ and $R_i$ make angle $\alpha_i$ and $\beta_i$ with the $X-$axis respectively. Let $T_i$ represent the point of intersection of the tangent at $R_i$ with the tangent at $P$. If the graph of $\hbox{Arc}(PR_i)$ is not a straight line then points $R_i,~T_i,~P$ are not collinear. The secant through $P$ and $R_i$ and the tangents at $P$ and $R_i$ all will coincide if the graph of $\hbox{Arc}(PR_i)$ is a straight line. This shows that the secant through $P$ and $R_i$ and the tangents at $P$ and $R_i$ all coincides if and only if the graph of $\hbox{Arc}(PR_i)$ is straight line. Now we see for any position of $R$ on $\hbox{Arc}(PQ)$ the tangent and the secant can never make an angle larger than $\theta$. Then as point $R$ tend towards $P$ we have,
\begin{align*}
&~~~~~~~~~~~~\alpha_1 < \alpha_2 < \cdots < \alpha_n < \cdots < \theta > \cdots > \beta_n > \cdots > \beta_2 > \beta_1 \RA \\  
&\tan \alpha_1 < \tan \alpha_2 < \cdots < \tan \alpha_n < \cdots < \tan \theta > \cdots \tan \beta_n > \cdots > \tan \beta_2 > \tan \beta_1.
\end{align*}
This implies that $\{ \alpha_n \}$ and $\{ \beta_n \}$ are increasing sequences with $\theta$ as the SUP and $\{ \tan \alpha_n \}$ and $\{ \tan \beta_n \}$ are increasing sequences with $\tan \theta$ as the SUP. Then as $n \ra \infty$ the secant-line represented by $\{ \alpha_n \}$ and the tangent-line represented by $\{ \beta_n \}$ tends to the tangent-line at $P$ implying $\hbox{Arc}(PR_n)$ tends to a straight line segment.  \hspace*{\fill} $\square$
\begin{thrm} \label{T2}
	For a cavex segment, the \emph{LUB} of the secant-measures, the \emph{GLB} of the tangent-measures and the arc-length of the cavex segment are all equal.
\end{thrm}
\begin{proof}
Let $f(x)$ be the cavex segment on $[a,b]$ and $f'(x)$ the slope at $x$. Assume that the arc length of $f(x)$ exists and let it be $\LL$. For $[a,b]$ create infinite partitions $\Omega_1 \subset \Omega_2 \cdots \subset \Omega_{n-1} \subset \Omega_n \cdots$ with their corresponding number of secant and tangent measures. From Lemma \ref{L2} for any partition $\Omega$, $S(\Omega) < T(\Omega)$. Then for any $r,s \in \N$, $S(\Omega_r)<T(\Omega_s)$, because for $r>s$, $S(\Omega_r)<T(\Omega_r)<T(\Omega_s)$ and for $r<s$, $S(\Omega_r)<S(\Omega_s)<T(\Omega_s)$. As $\{ S(\Omega_n) \}$ is bounded above by any of $T(\Omega)$'s and $\{ T(\Omega_n) \}$ is bounded below by any of $S(\Omega)$'s, LUB$(\{ S(\Omega_n) \})$ and $\hbox{GLB}(\{ T(\Omega_n) \})$ exists, then $\hbox{LUB}(\{ S(\Omega_n) \}) \leq \hbox{GLB}(\{ T(\Omega_n) \})$. For the partition $\Omega_n$ for $j=1, \ldots, n$, let $\hat{f}_j$ represent the segment of $f(x)$ on $[x_{j-1}, x_j]$ and $l_i$ be the arc-length of $\hat{f}_j$, then $\sum^n_{j=1} l_i=\LL$. Let for $j=0, \ldots, n$, $f(x_j)=y_j$. If the graph of $f(x)$ is a straight line on $I$ with length $\mathcal F$, then all $\hat{f}_j$'s are straight lines. The tangent chords to $f(x)$ at $(x_{j-1}, y_{j-1})$ and $(x_j, y_j)$ coincide with $\hat{f}_j$. Then any point on $\hat{f}_j$ can be taken as $(X_j,Y_j)$ and from \eqref{stm} we get $\F=S(\Omega_n)= T(\Omega_n) \RA \F=\hbox{LUB}(\{ S(\Omega_n) \})= \hbox{GLB}(\{ T(\Omega_n) \})$. Again from \eqref{stm} $t_j$ will be equal to $s_j$ only if $(X_j, Y_j)$ lies on the line joining $(x_{j-1}, y_{j-1})$ and $(x_j, y_j)$. This implies that the differences $|s_j-l_j|$, $|t_j-l_j|$, $|t_j-s_j|$ are identically zero if and only if $\hat{f}_j$ is a straight line, then a non-zero value of the differences will imply $\hat{f}_j$ is not a straight line. In the cavex segment for $j=1, \ldots, n$, we have $|s_j-l_j|$, $|t_j-l_j|$, $|t_j-s_j| \geq 0$. As $n \ra \infty$, $x_{j-1} \ra x_j$ and $y_{j-1} \ra y_j$ then from Lemma \ref{L3} each $\hat{f}_j$ tend to a straight line and hence each of the differences $|s_j-l_j|$, $|t_j-l_j|$, $|t_j-s_j|$ will tend to zero implying $\hbox{LUB}(\{S_n \}) = \hbox{GLB}(\{T_n \})=\LL$.	
\end{proof}
\noindent E.W. Hobson\cite{EW} shows the increasing sequence of perimeter of inscribed polygons is bounded above, L. Gillman\cite{LG} just defines $\C$ as the LUB of the perimeters of the inscribed polygons. Construction of a non-decreasing sequence which is bounded above just confirms the existence of a LUB and nothing else. In Theorem \ref{T2} it is the GLB of the tangent-measures which helps prove that the LUB is $\LL$. This shows that Euclidean metric is sufficient for the rectification of $f(x)$. In Theorem 2 the equality could also be achieved in the following way. Now suppose $\hbox{LUB}(\{S_n \}) = \hbox{GLB}(\{T_n \}) \neq \LL$ then there exists at least one point on a tangent-chord or/and a secant-chord between two adjacent points on $f(x)$ not lying on $f(x)$. Then even a single point refinement will increase $\hbox{LUB}(\{ S(\Omega_n) \})$ or/and decrease $\hbox{GLB}(\{ T(\Omega_n) \}) \RA \hbox{LUB}(\{ S(\Omega_n) \}) > \hbox{GLB}(\{ T(\Omega_n) \})$ a contradiction. So $\LL=\hbox{LUB}(\{ S(\Omega_n) \})= \hbox{GLB}(\{ T(\Omega_n) \})$.
\begin{cor} \label{C2} 
	The arc-length of a Cavex segment is rectifiable and it lies between the secant-measure and tangent-measure.
\end{cor}
\begin{proof} Let $\LL$ be the arc-length. From Theorem \ref{T2} for any cavex segment $\hbox{LUB}(\{ S(\Omega_n) \})= \hbox{GLB}(\{ T(\Omega_n) \})=\LL$ implying the arc-length of a Cavex segment is rectifiable, also we have $S(\Omega_n) < \LL < T(\Omega_n)~\fa~n$.
\end{proof}
\begin{cor} \label{C3}
	For a cavex segment the LUB of the secant-measures is the arc-length.
\end{cor}
\begin{proof}
	Since the segment is cavex, by Theorem \ref{T2} it is sufficient to find the LUB of secant-measures to get the value of the arc-length.
\end{proof}
\begin{cor}
	Among a chord and a cavex segment between two points, the chord length is smaller.
\end{cor}
\begin{proof}
	The secant-measure is the length of the chord. From Corollary \ref{C2} the secant-measure is strictly less than than the arc-length of a cavex segment.
\end{proof}
\noindent The above corollary justifies the proof of $\underset{x \ra 0}{\lim} \frac{\sin x}{x}=1$, that uses arc-length.
\begin{cor}
	Consider a chord and two non-intersecting concave/convex curves between two points. Then the curve that is contained inside the region bounded by the other curve and the chord is shorter.
\end{cor}

\begin{wrapfigure}{h}{0.42\textwidth}
	\centering	
	\includegraphics[width=0.38\textwidth, trim={8mm 0mm 0mm 4mm}]{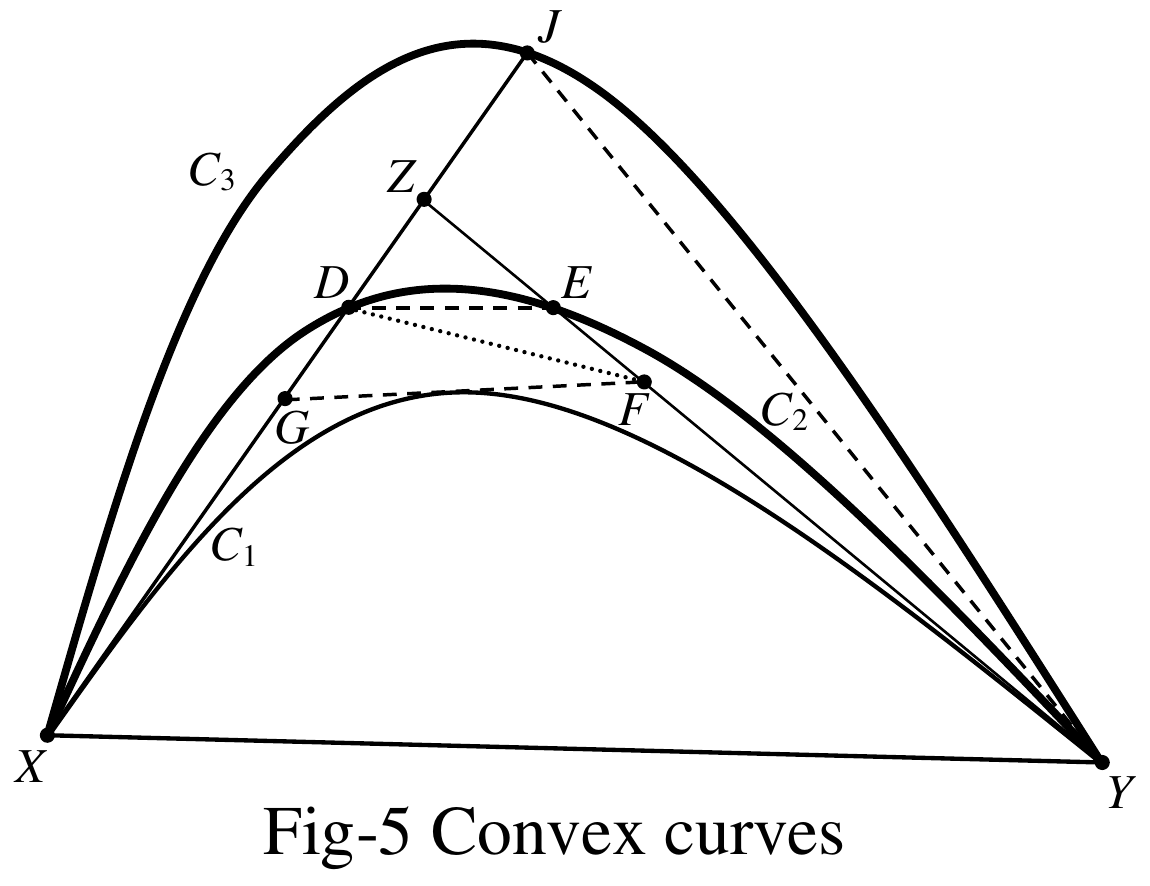}
\end{wrapfigure}

\noindent \ti{Proof}. With no loss of generality, convex curves $C_1,~C_2$ between points $X$ and $Y$ shown in Fig-5 represents the curves described in the corollary. Let $c_1,~c_2$ and $c_3$ be the arc-length of $C_1,~C_2$ and $C_3$ respectively. For $C_1$ draw tangent-chords at $X$ and $Y$ intersecting at $Z$. Let $D$ and $E$ be the points of intersection of $C_2$ with $\overline{XZ}$ and $\overline{ZY}$ respectively. Draw a tangent on $C_1$ intersecting $\overline{XZ}$ at $G$ and $\overline{ZY}$ at $F$. From triangle inequality we have $\overline{GD} + \overline{DF} > \overline{GF}$ and $\overline{DE} + \overline{EF} > \overline{DF} \RA \overline{GD} + \overline{DE} + \overline{EF}> \overline{GF}$. Then from Corollary \ref{C2} we have,
$$c_1< \overline{XG} + \overline{GF} + \overline{GY} < \overline{XG} + \overline{GD} + \overline{DE} + \overline{EF} + \overline{GY} <c_2.$$
The above inequality holds till $C_2$ passes through $Z$. For the convex curve $C_3$ going above $Z$, extend $\overline{XZ}$ to meet $C_3$ at $J$, then from Corollary \ref{C2} and triangle inequality
$$c_1 < \overline{XZ} + \overline{ZY} < \overline{XZ} + \overline{ZJ} + \overline{JY} < c_3. \eqno\square$$ 
\begin{cor} 
	The arc-length integral converges to the arc-length.
\end{cor}
\begin{proof}
	With no loss of generality let $[c, d]$ be the domain of the cavex 
segment of $f(x)$ differentiable on $(c,d)$. In Calculus, for defining 
arc-length all partitions taken are regular. For $n \in \N$, 
let $\Omega_n = \{ c = x_0 < x_1 < \ldots < 
x_{n-1} < x_n = d \}$ be a regular partition with $\Delta x$ as the sub-interval length, where $\Delta x \ra 0$ as $n \ra \infty$ and for $j =0, \ldots , n$, let $y_j = 
f(x_j)$. For $j =1, \ldots , n$, by mean value theorem there exists a $\bar{x}_j \in (x_{j-1}, x_j)$ 
such that $(y_j-y_{j-1}) = f'(\bar{x}_j)(x_j-x_{j-1})$, where 
$f'(\bar{x}_j)=(df/dx)|_{_{\bar{x}_j}}$. Then substituting for $(y_j-y_{j-1})$ in the following sum,
	\begin{equation}
	\underset{j=1}{\overset{n}{\sum}} \sqrt{(x_j-x_{j-1})^2 + (y_j-y_{j-1})^2 }. \label{al}
	\end{equation}
and taking the limit $n \ra \infty$ gives
	\begin{equation}
	\underset{n \ra \infty}{\lim} \underset{j=1}{\overset{n}{\sum}} \sqrt{1+[f'(\bar{x}_j)]^2}~ \Delta x_n = \overset{d}{\underset{c}{\int}} \sqrt{1+[f'(x)]^2} ~dx, \label{alf}
	\end{equation}
	defined as the arc-length integral, where $\Omega_1 \subset \Omega_2 \cdots \subset \Omega_{n-1} \subset \Omega_n \cdots$ are infinite regular partitions on $[c,d]$ with $\Delta x \ra 0$ as $n \ra \infty$. Since \eqref{al} is the secant-measure of the arc-length of $f(x)$ on $[c,d]$ for $\Omega_n$, \eqref{alf} will be the LUB of the secant-measures hence by Corollary \ref{C3} is the arc-length.
\end{proof} 
\begin{cor} \label{C7}
	$\C=2 \pi r$.
\end{cor}
\begin{proof}
	From Theorem \ref{T1} we have $\hbox{LUB}(\{ P_n^L \})= \hbox{GLB}(\{ P_n^U \})=2 \pi r$. Since $P_n^L$ is secant-measure and $P_n^U$ is tangent-measure, from Theorem \ref{T2} $\hbox{LUB}(\{ P_n^L \})= \hbox{GLB}(\{ P_n^U \})=\C \RA \C=2 \pi r$.
\end{proof}
\begin{cor}
	The statements ``The ratio $\A/r^2$ is a constant, independent of $r$", and ``The ratio $\C/2r$ is a constant, independent of $r$" are equivalent.
\end{cor}
\begin{proof}
	From Theorem \ref{T1} and Corollary \ref{C7}, $\pi=\A/r^2=\C/2r$.
\end{proof}
\noindent The above corollary shows the equivalence of the classical definitions (i) and (ii).\\\\
\tb{Acknowledgements:} I sincerely thank, Mr. M.A.Prasad for many fruitful discussions and suggestions. Heartfelt thanks to Prof. V.Vetrivel and the reviewer for their critical comments.

\end{document}